\documentclass[12pt,oneside]{article}
\usepackage{amsmath}
\usepackage{amsfonts}
\usepackage{amssymb}
\usepackage{graphicx}

\usepackage{latexsym}
\usepackage{amssymb, latexsym}
\usepackage{amsmath}
\usepackage{mathrsfs}
\usepackage{amsxtra}
\usepackage{amsthm}
\usepackage{setspace}
\usepackage{epsfig}
\usepackage{geometry}
\usepackage{graphicx}
\usepackage{blindtext}

\newtheorem{theorem}{Theorem}[section]
\newtheorem{corollary}[theorem]{Corollary}

\newtheorem{lemma}[theorem]{Lemma}

\theoremstyle{definition}

\newtheorem{remark}[theorem]{Remark}

\date{}

\begin{document}

\author{Sergei Artamoshin
\thanks{The author wants to thank professor Adam Kor\'anyi, who pointed out that $\omega^{\alpha}(x,y)$ originally studied by the author, is an eigenfunction of the hyperbolic Laplacian. The author is also thankful to professor J\a'ozef Dodziuk for his interest to this paper and for his comments.}\\ \\Department of Mathematics, University of Massachusetts \\Dartmouth, USA\\ \\sartamoshin@gmail.com%
}

\title{\textbf{At Most Two Radii Theorem For A Real Eigenvalue Of The Hyperbolic Laplacian}}

\maketitle

\begin{abstract}
We study a $(k+1)$-dimensional hyperbolic space of a negative constant sectional curvature $\kappa=-1/\rho^2$. Let $\lambda$ be a real eigenvalue and $f_{\lambda} (x)$ be an eigenfunction of the hyperbolic Laplacian assuming a non-zero value at $x_0$. Then the average value of $f_{\lambda}(x)$ over any sphere centered at $x_0$ allows to identify the corresponding eigenvalue $\lambda$ uniquely as long as that average value is large enough. Otherwise, to identify the corresponding eigenvalue uniquely, we need to make sure that the computed average value is not zero and then we need to compute an additional average value of $f_{\lambda}(x)$ over a small enough sphere centered at the same point $x_0$.

\end{abstract}

\bigskip

\begin{spacing}{1.5}
{%\Large

\section{Introduction}

To start, let us clarify the difference between the statement of the Two-Radius theorem given in this paper and the statement of the famous Delsarte's Two-Radius Theorem for harmonic functions in $R^n$, see~\cite{Flatto} or~\cite{Delsarte}. Delsarte's theorem uses two radii at every point in $R^n$ to conclude that the given function is harmonic. The two radii theorem derived below, assumes that the given function is an eigenfunction and uses at most two radii just at a single point to compute the eigenvalue. Thus, it is an open question, if we can obtain a statement similar to Delsarte's theorem. Say, we used two radii at every point of the hyperbolic space and obtained the same eigenvalue. Does it mean that the given function is an eigenfunction? This question remains unanswered in this paper. It is interesting that in the case of a harmonic function in a hyperbolic space, only one radius at a single point is necessary to conclude that the eigenvalue must be zero.

All the derivations we shall see below, are based on the explicit representation of a radial eigenfunction. One of such possible representations is given by hypergeometric function, for example, see A.Kor\'anyi~\cite{Koranyi}, p. 111, formula~(3.10). For the computational reasons, we are going to use an alternative approach based on the geometric technique originally introduced by Hermann Schwarz in 1890, see~\cite{Schwarz} (pp. 359-361) or~\cite{Ahlfors} (pp. 168). He found that Poisson Kernel in a two dimensional space can be represented as a ratio of two segments forming one chord in a circle. A bit later, in 2005, Adam Kor\'anyi pointed out that the same ratio of two segments, being rased to any power, represents an eigenfunction of the hyperbolic Laplacian. Combination of these two ideas yields the technique that was already described in~\cite{Artamoshin} and applied to estimate the lower eigenvalue of the hyperbolic Laplacian. In this paper we shall see another application of the technique for analysis of radial eigenfunctions.

First, we may observe that by the uniqueness of Haar measure, the average value of every eigenfunction $f_\lambda(x)$ over a geodesic sphere $S^k(x_0,r)$ of radius $r$ centered at a point $x_0$ defined as

\begin{equation}
     \varphi_\lambda(r)=\frac{1}{|S^k(r)|}\int\limits_{S^k(x_0,r)}f_\lambda(x)dS_x
\end{equation}
is a radial eigenfunction with the same eigenvalue $\lambda$.

Second, using the proper geodesic hyperplane mirror we can isometrically reflect $x_0$ onto the origin of the ball model of a hyperbolic space. Such a transformation does not change eigenvalue, and therefore, allows us to reduce any radial eigenfunction to the radial eigenfunction centered at the origin.

Third, since the eigenvalue $\lambda$ is independent on a constant $k$ in the equation $\vartriangle (kf)+\lambda (kf)=0$, we may assume that the chosen eigenfunction with a non-zero value at the point of radialization $x_0$,  just equals to $1$ at $x_0$.

Therefore, uniting the above arguments, we can say that $x_0=O$, $f(O)=1$ and then the radial eigenfunction with the same eigenvalue $\lambda$ in the ball model can be obtained by the following Euclidean integral.

\begin{equation}
     \varphi_\lambda(r)=\frac{1}{|S^k(\eta)|}\int\limits_{S^k(\eta)}f_\lambda(m) dS_m,
\end{equation}
where $|m|=\eta=\rho\tanh(r/(2\rho))$ is the Euclidean radius of the geodesic sphere $S^k(r)$, and $S^k(\eta)$ is the Euclidean sphere of radius $\eta$. Thus investigation of averages can be reduced to the investigation of radial functions centered at the origin.

Finally, we shall see that every radial eigenfunction centered at the origin and assuming the value one at the origin, can be represented explicitly in terms of the geometric interpretation of Poisson kernel, originally introduced by Herman Schwarz.

Let $B^{k+1}_\rho$ be the ball model of a hyperbolic $(k+1)$-dimensional space with a negative constant sectional curvature $\kappa=-1/{\rho^2}$. Consider the set of all radial eigenfunctions of the Hyperbolic Laplacian assuming the value $1$ at the origin, i.e. the set of all solutions for the following system
\begin{equation}\label{System/DE/for/RE}
\left\{
  \begin{array}{ll}
     & \hbox{$\varphi^{''}(r)+\frac{k}{\rho}\coth\left(\frac{r}{\rho}\right)\varphi^{'}(r)+\lambda\varphi(r)=0$, $\lambda\in\mathbf{C}$;} \\
     & \hbox{$\varphi(0)=1$\,,}
  \end{array}
\right.
\end{equation}
written in the geodesic polar coordinates of the hyperbolic space.
Recall that for every $\lambda\in\mathbf{C}$ there exists a unique solution $\varphi_\lambda(r)$ such that $\varphi_\lambda(0)=1$, see \cite{Chavel} (p. 272).

In this paper we shall see that a real value of a radial eigenfunction at an arbitrary point $r_0$ allows us to restore the unique real eigenvalue $\lambda$ and therefore, the unique real eigenfunction $\varphi_\lambda(r)$ as long as
\begin{equation}\label{BoundaryFcn}
    \varphi_\lambda(r_0)\geq V(\eta(r_0))=\frac{1}{|S^k(\rho)|}\int\limits_{S^k(\rho)}
    \left( \frac{\rho^2-\eta^2}{|u-m|^2} \right)^{k/2} dS_u\,,
\end{equation}
where $S^k(\rho)$ is the $k$-dimensional sphere of radius $\rho$ centered at the origin and $m$ can be any point from $S^k(\eta)$ with $\eta=\eta(r_0)=\rho\tanh(r_0/(2\rho))$.

In this case we may conclude that $\lambda\leq-\kappa k^2/4$, $\varphi_\lambda(r)$ never vanishes, and $V(\eta(r_0))\rightarrow 0$ as $r_0\rightarrow\infty$. On the other hand, we shall see that according to Picard's Great Theorem, there exists infinitely many complex valued eigenvalues with complex valued eigenfunctions assuming the same value at the same point $r_0$ as $\varphi_\lambda(r)$ assumes at $r_0$.

If the inequality in \eqref{BoundaryFcn} fails, we still can restore the corresponding eigenvalue $\lambda$ uniquely, but need to make sure that $\varphi(r_0)\neq 0$ and use an additional value of $\varphi_\lambda(r_1)$ chosen at a point $r_1$ close enough to the origin. The upper bound for $r_1$ depends on the value of $\varphi_\lambda(r_0)$ and the dimension of a hyperbolic space.

\section{Preliminaries and Notations.}

We consider the Ball model of the $(k+1)$-dimensional Hyperbolic space of a negative constant sectional curvature $\kappa=-1/\rho^2$. In this model, the whole hyperbolic space is represented by the open ball $B^{k+1}(\rho)$ of radius $\rho$ and centered at the origin. Every point $m\in B^{k+1}(\rho)$ must be represented as $m=(X,T)=(X_1, X_2, ...X_k, T), |X|^2+T^2<\rho^2,$ in $R^{k+1}$, and the metric is given by $[2\rho^2/(\rho^2-|X|^2-T^2)]^2|ds|^2$. Let $\eta=|m|$ be the Euclidean distance between the origin and point $m$, while $r$ be the hyperbolic distance between the same points. For the notation a reader may refer to figure~\ref{BallSpacePortrait}.

\begin{figure}[h]
    \center\epsfig{figure=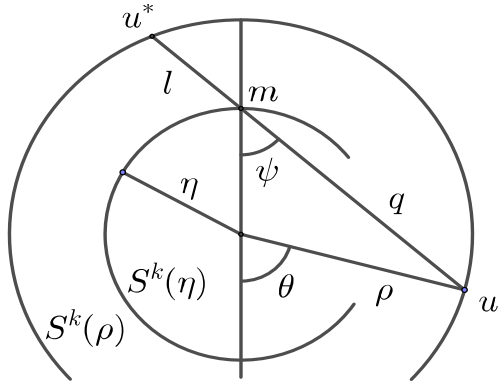, height=5cm, width=12cm}
  \caption{Euclidean variables in the hyperbolic disc.}
  \label{BallSpacePortrait}
\end{figure}
%\begin{figure}[h]
%    \center\epsfig{figure=Expli_Sol_Extra_New.eps, height=5.5cm, width=7cm}
%  \caption{Euclidean variables in the hyperbolic disc.}
%  \label{BallSpacePortrait}
%\end{figure}
Recall that $S^k(\rho)$ is the $k$-dimensional sphere of radius $\rho$ centered at the origin,
    \begin{equation}\label{Eucli-Hype/Variables}
        \eta=\rho\tanh\left( \frac{r}{2\rho} \right)
        \quad\text{and}\quad
        r=\rho\ln\frac{\rho+\eta}{\rho-\eta},
    \end{equation}
and let $u$ be an arbitrary point from $S^k(\rho)$. Now we are ready for the list of statements that we are going to use as future references.

\begin{theorem}\label{Thm-General_Form/Of/RE/and/Identity} Every radial eigenfunction, complex or real valued, $\varphi_\lambda(r)$ assuming the value one at the origin, must be unique for a given eigenvalue $\lambda$, and can be represented as one of the following integrals,
    \begin{equation}\label{General_Form/Of/RE/and/Identity}
        \varphi_\lambda(r)
        =\frac{1}{|S^k(\rho)|}\int\limits_{S^k(\rho)}\omega^\alpha(m,u)dS_u
        =\frac{1}{|S^k(\rho)|}\int\limits_{S^k(\rho)}\omega^{k-\alpha}(m,u)dS_u,
    \end{equation}
where $\alpha$ is any of the two roots of the quadratic equation
    \begin{equation}\label{Alpha-Lambda-Equation/And/Omega}
        \lambda=-\kappa(\alpha k-\alpha^2),
        \quad\text{and}\quad
        \omega(m,u)=\frac{\rho^2-\eta^2}{|m-u|^2}.
    \end{equation}
\end{theorem}

\begin{remark}\label{Rmk-Lambda/Real/Conditoin}
    An elementary computation shows that the quadratic equation
    in~\eqref{Alpha-Lambda-Equation/And/Omega} implies that an eigenvalue $\lambda$ is real if and only if $\alpha$ is real or $\alpha=k/2+ib$. Note that every real $\alpha$ generates a never vanishing radial eigenfunction with $\lambda\leq -\kappa k^2/4$. For $\alpha=k/2+ib$ with $b\neq0$, the corresponding $\lambda$ must be strictly greater than $-\kappa k^2/4$.
\end{remark}

\begin{proof}[Proof of theorem~\ref{Thm-General_Form/Of/RE/and/Identity}.] This theorem is obtained as the combination of theorem 3$\cdot$1$\cdot$1 and proposition 3$\cdot$1$\cdot$2 from~\cite{Artamoshin}.
\end{proof}

\begin{lemma}[Geometric Interpretation of $\omega$.]\label{Lemma-Geo-Interpre-Omega} $\quad$

    (A) Let $\eta<\rho$, and $m,u$ be two arbitrary points such that $m\in S^k(\eta)$, $u\in S^k(\rho)$. We define the point $u^*$ as the intersection of line $mu$ and sphere $S^k(\rho)$. Then the function $\omega=\omega(m,u)$ defined in \eqref{Alpha-Lambda-Equation/And/Omega} can be represented as the ratio of two segments, $\omega=l/q$, where $q=|m-u|$ and $l=|m-u^*|$. For the notation refer the Figure~\ref{BallSpacePortrait}.

    (B) If we define the angle $\psi=\angle umO$, then
    \begin{equation}\label{Deriva-of-Omega/and/(l+q)}
        d\psi=\frac{l+q}{4\eta\sin\psi}d\ln\omega
        \quad\text{and}\quad
        \frac{d(l+q)}{d\psi}=-\frac{2\eta^2\sin(2\psi)}{l+q}.
    \end{equation}
\end{lemma}

\begin{remark}
    The ratio of two segments $l/q$ was originally introduced by Hermann Schwarz and called as a geometric interpretation of $\omega$, see~\cite{Schwarz} (pp. 359-361) or~\cite{Ahlfors} (p. 168).
\end{remark}

\begin{proof}[Proof of lemma~\ref{Lemma-Geo-Interpre-Omega}.]
    Statement (A) of the lemma together with the first identity in\eqref{Deriva-of-Omega/and/(l+q)} can be obtained as the combination of theorem~2$\cdot$1$\cdot$1 and lemma~2$\cdot$4$\cdot$1 from~\cite{Artamoshin}. To obtain the second formula in~\eqref{Deriva-of-Omega/and/(l+q)}, we apply Pythagorean theorem and observe that $(l+q)^2/2=\rho^2-\eta^2\sin^2\psi$.
%    $$\left( \frac{l+q}{2} \right)^2=\rho^2-\eta^2\sin^2\psi.$$
    Differentiation of the last formula with respect to $\psi$ completes the proof.
\end{proof}

\begin{theorem}[A vanishing radial eigenfunction]\label{Thm-Vanishing-RE-Final}
       Let $\lambda$ be real and $\varphi_\lambda(r)$ is a radial eigenfunction assuming value $1$ at the origin. Then $\varphi_\lambda(r)=0$ at some $r<\infty$ if and only if $\lambda>-\kappa k^2/4$ or equivalently,
       \begin{equation}\label{Lambda-Alpha-k-Thm}
            \lambda=\frac{\alpha k-\alpha^2}{\rho^2}
            \quad\text{with}\quad
            \alpha=\frac{k}{2}+ib
            \quad\text{and}\quad
            b=\sqrt{\frac{\lambda}{-\kappa}-\frac{k^2}{4}}>0.
       \end{equation}
In this case the eigenfunction can be written as
    \begin{equation}\label{Vanishing-RE-General}
        \varphi_\lambda(r(\eta))
        =\frac{1}{|S^k(\rho)|}\int\limits_{S^k(\rho)}\omega^\alpha dS_y
        =\frac{1}{|S^k(\rho)|}\int\limits_{S^k(\rho)}\omega^{k/2}\cos(b\ln\omega) dS_y\,,
    \end{equation}
or equivalently,
    \begin{equation}\label{Vanishing-RE-Final}
        \varphi_\lambda(r(\eta))
        =\frac{4\rho(\rho^2-\eta^2)^{k/2}\sigma_{k-1}}{|S^k(\rho)|}\int\limits_0^{\pi/2}
        \frac{\sin^{k-1}\psi}{l+q}\cos\left(b\ln\frac{l}{q}\right)d\psi\,.
    \end{equation}
\end{theorem}

\begin{proof}
    If $\varphi_\lambda(r)=0$ at some $r<\infty$, then all the consequences stated in theorem can be obtained as the combination of theorem~3$\cdot$2$\cdot$1 and lemma~3$\cdot$2$\cdot$1 from~\cite{Artamoshin}. We just need to proof that if $\lambda>-\kappa k^2/4$, then every radial eigenfunction must be vanishing at some finite point $r$. In fact, such an eigenfunction has infinitely many zeros. Recall from~\eqref{System/DE/for/RE} that $\varphi_\lambda(r)$ is a solution for the following differential equation
    \begin{equation}\label{radial-laplac-Recall}
    Y''+\frac k\rho\coth\left(\frac r\rho\right)Y'+\lambda Y=0.
\end{equation}
The direct computation shows that the substitution $Y=U(\sinh(r/\rho))^{-k/2}$ suggested in~\cite{Simmons}, p.~119,   reduces~\eqref{radial-laplac-Recall} to
$U''+QU=0$, where
%qqqhhhhhhhhhhhhh
\begin{equation}\label{Expression-q-after-substitution}
    Q=Q(r)=\lambda-\frac{k^2}{4\rho^2}-\frac{k(k-2)}{4\rho^2\sinh^2(r/\rho)}
\end{equation}
It is clear that
\begin{equation}\label{limit-for-q(r)}
    \lim\limits_{r\rightarrow\infty}Q(r)=\lambda+\frac{\kappa k^2}{4}\,,\quad\text{where}\quad \kappa=-\frac{1}{\rho^2}\,.
\end{equation}
Therefore, for $\lambda>-\kappa k^2/4$ there exist positive numbers $\tau$ and $\Lambda$ such that
\begin{equation}\label{Inequalities-for-q-on-large-r}
    Q(r)>\tau>0\quad\text{for every}\quad r>\Lambda.
\end{equation}
Let us introduce $Z(r)$ as a solution of the following elementary equation $Z''+\tau Z=0$. According to the Sturm comparison theorem in \cite{Simmons}, p.~122, condition \eqref{Inequalities-for-q-on-large-r} implies that $U(r)$ vanishes at least once between any two successive zeros of $Z(r)$. It clear that $Z(r)$ has infinitely many zeros. Therefore, $U(r)$ as well as $\varphi_\lambda(r)$ must also have infinitely many zeros as $r\rightarrow\infty$. In addition, as a consequence of Theorem~\ref{Infinity-Behavior-radial-Theorem}, p.~\pageref{Infinity-Behavior-radial-Theorem}, we shall see that $\varphi_\lambda(r)$ vanishes at $\infty$.
\end{proof}

%\begin{remark}
%    For every $\lambda>-\kappa k^2/4$ there exist three positive constants $T_1, T_2$, and $\Lambda_0$ such that $T_1<|r_2-r_1|<T_2$ for every two successive zeros $r_1, r_2$ of $\varphi_\lambda(r)$ satisfying $r_2>r_1>\Lambda_0$.
%\end{remark}

\begin{theorem}\label{Infinity-Behavior-radial-Theorem}
    If $\varphi_\lambda(r)$ is a real radial eigenfunction assuming value $1$ at the origin, then
\begin{equation}\label{Infinity-Behavior-radial-Formula}
\lim\limits_{r\rightarrow\infty}\varphi_\lambda(r)=
\left\{
  \begin{array}{ll}
    \infty, & \hbox{if $\quad \lambda<0$;} \\
    1, & \hbox{if $\quad \lambda=0$;} \\
    0, & \hbox{if $\quad \lambda>0$.}
  \end{array}
\right.
\end{equation}
\end{theorem}

\begin{proof}
Recall that $\varphi_\lambda(r)$ is a solution of~\eqref{radial-laplac-Recall}. As we already saw above, the direct computation shows that the substitution $Y=U(\sinh(r/\rho))^{-k/2}$ suggested in~\cite{Simmons}, p.~119,   reduces~\eqref{radial-laplac-Recall} to
$U''+QU=0$, where $Q=Q(r)$ is given in~\eqref{Expression-q-after-substitution}. Therefore,
$\varphi_\lambda(r)$ satisfies the following differential equation
$$\left[(\sinh(r/\rho))^{k/2}\varphi_\lambda(r)\right]''=
\left(\frac{k^2}{4\rho^2}-\lambda+\frac{k(k-2)}{4\rho^2\sinh^2(r/\rho)}\right)
\left[(\sinh(r/\rho))^{k/2}\varphi_\lambda(r)\right],$$
which can be investigated by the standard theory of ODE. Theorem~2.1 from \cite{Olver}, p.~193 yields the behavior of $\varphi_\lambda(r)$ for $k^2/4\rho^2>\lambda$ and $r\rightarrow\infty$. Theorem~2.2 from~\cite{Olver}, p.~196 yields the behavior of $\varphi_\lambda(r)$ for $k^2/4\rho^2<\lambda$. To describe the behavior of $\varphi_\lambda(r)$ for $\lambda=k^2/4\rho^2$ it is enough to apply Corollary~9.1 from~\cite{Hartman}, p.~380.
\end{proof}

 Let us fix $r$ and consider $\varphi_\lambda(r)$ as a function of $\alpha$. Recall Picard's Great Theorem
\begin{theorem}[Picard's Great Theorem]\label{Picard-Great-theorem}
 Let $c\in\mathbb{C}$ be an isolated essential singularity of $F$. Then, in every neighborhood of $c$, $F$ assumes every complex number as a value with at most one exception infinitely many times, see \cite{Reinhold}, (p.240).
\end{theorem}

\begin{lemma}\label{Picard-Application} For every $\beta\in \mathbb{C}$ there are infinitely many numbers $\alpha\in\mathbb{C}$ such that
\begin{equation}\label{Deinitoin-F(alpha)/in/Picard}
    F(\alpha)\equiv\int\limits_{S^k(\rho)}\omega^\alpha dS_u=
    \int\limits_{S^k(\rho)}\omega^\beta dS_u\,.
\end{equation}
\end{lemma}

 \begin{proof}[Proof of Lemma~\ref{Picard-Application}] This lemma can be obtained as a consequence of Picard's Great Theorem stated above. Indeed, the direct computation shows that $F$ is an entire function, which means that $\infty$ is an isolated singularity of $F$. To see that $\infty$ is the essential singularity for $F(\alpha)$, let us obtain the Taylor decomposition of $F$ at the point $\alpha=k/2$. Observe that

\begin{equation}\label{Integral-Id-with/pm(ib)}
    \int\limits_{S^k(\rho)}\omega^{k/2+ib}dS_u=
    \int\limits_{S^k(\rho)}\omega^{k/2-ib}dS_u\,.
\end{equation}
Comparing the imaginary parts of this identity we can observe that both of them must be identically equal to zero and then, for $\alpha=k/2+ib$ we have
    \begin{equation}\label{Integral/without/Imaginary}
        F(\alpha)=\int\limits_{S^k(\rho)}\omega^\alpha dS_u=
        \int\limits_{S^k(\rho)}\omega^{k/2}\cos(b\ln\omega) dS_u\,.
    \end{equation}
The repeated differentiation with respect to $z=k/2+ib$ yields
    \begin{equation}\label{Taylor-Deco-Proof-2}
        F^{(2m+1)}(k/2)=0 \quad\text{and}\quad
        F^{(2m)}(k/2)=\int\limits_{S^k(\rho)}\omega^{k/2}(\ln\omega)^{2m}dS>0
    \end{equation}
for $m=0,1,2,\cdots$. Therefore,
\begin{equation}\label{Taylor-Decomposition}
    F(\alpha)=\int\limits_{S^k(\rho)}\omega^{\alpha}dS_u
    =\sum\limits_{m=0}^{\infty}\frac{(\alpha-k/2)^{2m}}{(2m)!}
    \int\limits_{S^k(\rho)}\omega^{k/2}(\ln\omega)^{2m}dS_u\,.
\end{equation}
This Taylor decomposition shows that $\infty$ is the essential singularity for $F(\alpha)$, since all even Taylor coefficients are positive. It is clear also, we can not meet an exception mentioned in Picard's theorem, since we are looking for complex numbers $\beta$, satisfying $F(\beta)=F(\alpha)$ for a given $\alpha$. This means that the value $F(\alpha)$ is already assumed by $F$, and then, by Picard's Great Theorem, $F$ must attain this value infinitely many times in every neighborhood of an isolated essentially singular point, i.e., at every neighborhood of $\infty$, in our case. This completes the proof of Lemma~\ref{Picard-Application}.
\end{proof}

\section{One Radius theorem for a real radial eigenfunction.}

\begin{theorem}[One Radius Theorem for Large Real Eigenfunctions]\label{OneRad_LargeValues}
    If $\varphi_\lambda(r)$ is a real radial eigenfunction assuming value $1$ at the origin and satisfying inequality in~\eqref{BoundaryFcn} at a particular point $r$, then the corresponding real eigenvalue can be uniquely restored.
\end{theorem}

\begin{proof}
    Recall from theorem~\ref{Thm-General_Form/Of/RE/and/Identity} that $\lambda=-\kappa(\alpha k-\alpha^2)$ and every radial eigenfunction $\varphi_\lambda(r)$ can be represented by any of the integrals in~\eqref{General_Form/Of/RE/and/Identity}. This $\lambda$, by the remark~\ref{Rmk-Lambda/Real/Conditoin}, is real if and only if $\alpha$ is real or $\alpha=k/2+ib$. In the case if $\alpha$ is real, the second derivative of the first integral in~\eqref{General_Form/Of/RE/and/Identity} yields
    \begin{equation}\label{Real Convex}
        \frac{d^2\varphi_\lambda(r)}{(d\alpha)^2}=\frac{1}{|S^k(\rho)|}\int\limits_{S^k(\rho)}
        \omega^{\alpha}(m,u)(\ln\omega)^2 dS_u>0\,.
    \end{equation}
    The last inequality together with the symmetry of the integral with respect to $\alpha=k/2$, provided by the integral identity~\eqref{General_Form/Of/RE/and/Identity}, implies that for every fixed $r$, $\varphi_\lambda(r)$ is strictly decreasing as a function of $\lambda\in(-\infty, -\kappa k^2/4]$ and achieves its minimum at $\alpha=k/2$. Therefore, the inequality in~\eqref{BoundaryFcn} holds for every real $\alpha$ or equivalently, for every $\lambda\leq-\kappa k^2/4$. To complete the proof, we just have to show that inequality in formula~\eqref{BoundaryFcn} fails for every $\lambda>-\kappa k^2/4$. Note that for such $\lambda$'s, by theorem~\ref{Thm-Vanishing-RE-Final}, $\varphi_\lambda(r)$ can be represented by the last integral formula in~\eqref{Vanishing-RE-General}, which allows us to write the following sequence of inequalities.
    \begin{equation}\label{EF-Inequality-Sequance/for/Alpha-k-b}
        \frac{1}{|S^k(\rho)|}\int\limits_{S^k(\rho)}
        \omega^\alpha dS_u
        \geq\frac{1}{|S^k(\rho)|}\int\limits_{S^k(\rho)}
        \omega^{k/2} dS_u
        >\frac{1}{|S^k(\rho)|}\int\limits_{S^k(\rho)}
        \omega^{k/2}\cos(b\ln\omega) dS_u\,,
    \end{equation}
    where the last integral formula was obtained in theorem~\ref{Thm-Vanishing-RE-Final} and represented a real radial eigenfunction with $b\neq0$ and vanishing at some finite point, while the first two integral formulae represent a never vanishing real radial eigenfunction $\varphi_\lambda(r)$. It is clear that the first inequality in~\eqref{EF-Inequality-Sequance/for/Alpha-k-b} holds for every point $r$ and real $\alpha$ or equivalently, for every $\lambda\leq-\kappa k^2/4$. The second inequality holds at every point $r$ and for every $b\neq 0$ or equivalently, for every $\lambda\geq-\kappa k^2/4$. This completes the proof of theorem~\ref{OneRad_LargeValues}.
\end{proof}

%\begin{corollary}
%    If a real radial eigenfunction $\varphi_\lambda(r)$ fails the inequality in~\ref{BoundaryFcn}, then $\lambda>-\kappa k^2/4$.
%\end{corollary}

%\begin{proof}[Proof of Corollary]
%    As we just saw in the proof of theorem~\ref{OneRad_LargeValues}, the first inequality in~\eqref{EF-Inequality-Sequance/for/Alpha-k-b} holds for every real $\alpha$, which corresponds to $\lambda\leq-\kappa k^2/4$, since $\lambda=-\kappa(\alpha k-\alpha^2)$. Threfore, if the inequality in~\eqref{BoundaryFcn} fails, we have to conclude that $\lambda>-\kappa k^2/4$.
%    Recall that the last quadratic equation yields real $\lambda$'s also for $\alpha=k/2+ib$ with $b\neq0$, which yields $\lambda>-\kappa k^2/4$. In this case, by theorem~\ref{Thm-Vanishing-RE-Final}, a radial eigenfunction must be represented by the last integral formula in~\eqref{EF-Inequality-Sequance/for/Alpha-k-b}, and therefore, the inequality in~\eqref{BoundaryFcn} fails.
%\end{proof}

\begin{remark}
    Consider radial eigenfunctions satisfying the system~\eqref{System/DE/for/RE}. The second integral formula in~\eqref{EF-Inequality-Sequance/for/Alpha-k-b} represents a never vanishing radial eigenfunction that separates never vanishing eigenfunctions and eigenfunctions assuming zero at some finite point. According to theorem~\ref{Infinity-Behavior-radial-Theorem}, this separation function is vanishing at infinity.
\end{remark}

Note that the argument used in the proof of theorem~\ref{OneRad_LargeValues} lead us to the following statement.

\begin{corollary}\label{Cor-Separator-Lambdas}
    Consider a real radial eigenfunctions $\varphi_\lambda(r)$ satisfying the system~\eqref{System/DE/for/RE}.

    If such an eigenfunction satisfies the inequality in~\eqref{BoundaryFcn}, then $\lambda\leq-\kappa k^2/4$ and $\alpha$ is real.

    If $\varphi_\lambda(r)$ fails the inequality in~\eqref{BoundaryFcn}, then $\lambda>-\kappa k^2/4$ and $\alpha=k/2+ib$ with $b\neq0$.
\end{corollary}

To analyse the uniqueness of the corresponding eigenvalue in the case if radial eigenfunction fails the inequality in \eqref{BoundaryFcn}, we need to obtain the upper bound for the corresponding real eigenvalue based on a non-zero value of $\varphi_\lambda(r)$, or equivalently, an upper bound for the parameter $b$ in the decomposition of $\alpha=k/2+ib$.

\begin{theorem}[An Upper Bound for $\lambda$]\label{Thm-UpperBoundEgenvalue}
     Let a real radial eigenfunction assumes the value $1$ at the origin and $\varphi_\lambda(r)\neq 0$ for some point $r$. Let the eigenfunction fails the inequality in \eqref{BoundaryFcn}. Then, the corresponding eigenvalue together with the parameter $b$ can be estimated from above as follows.

If $k=1$, then
\begin{equation}
        |b|\leq\max\left\{\frac{2}{\pi}, \frac{T_1(r,\kappa)}{\varphi^2_\lambda(r)}\right\}
        \quad\text{and}\quad
        \lambda\leq -\frac{\kappa k^2}{4}-\kappa \left\{\max\left\{\frac{2}{\pi}, \frac{T_1(r,\kappa)}{\varphi_\lambda^2(r)}\right\}
        \right\}^2,
\end{equation}
where $T_1(r,\kappa)=(9r)(\rho^2-\eta^2)/(\rho\eta^2\pi^2)$.

     If $k=2$, then
    \begin{equation}
        |b|\leq\frac{T_2(r,\kappa)}{|\varphi_\lambda(r)|}\quad\text{and}\quad
        \lambda\leq-\frac{\kappa k^2}{4}-\kappa \left(\frac{T_2(r,\kappa)}{|\varphi_\lambda(r)|}\right)^2,
    \end{equation}
where $T_2(r,\kappa)=1/\sinh(r/\rho)$.

If $k\geq3$, then
\begin{equation}
    |b|\leq\frac{T_3(\kappa,k,r)}{|\varphi_\lambda(r)|}
    \quad\text{and}\quad
    \lambda\leq -\frac{\kappa  k^2}{4}-\kappa\left(\frac{T_3(\kappa,k,r)}{|\varphi_\lambda(r)|}\right)^2\,,
\end{equation}
where $T_3=(k-2)\pi\rho(\rho^2-\eta^2)^{k/2}\sigma_{k-1}/(2\eta|S(\rho)|)$ and $|S(\rho)|$ is the volume of $k$-dimensional sphere.

\end{theorem}

\begin{proof}
    Since $\varphi_\lambda(r)$ fails the inequality in~\eqref{BoundaryFcn}, we conclude by corollary~\ref{Cor-Separator-Lambdas} that $\alpha=k/2+ib$ with $b\neq 0$, or equivalently, $\lambda>-\kappa k^2/4$. In this case, by theorem~\ref{Thm-Vanishing-RE-Final}, the eigenfunction can be represented as % given in~\eqref{Vanishing-RE-Final}
    \begin{equation}\label{Vanishing-RE-Final/Repro}
        \varphi_\lambda(r(\eta))
        =\frac{4\rho(\rho^2-\eta^2)^{k/2}\sigma_{k-1}}{|S(\rho)|}\int\limits_0^{\pi/2}
        \frac{\sin^{k-1}\psi}{l+q}\cos\left(b\ln\frac{l}{q}\right)d\psi\,,
    \end{equation}
with $b$ defined in~\eqref{Lambda-Alpha-k-Thm}. Refer to Figure~\ref{BallSpacePortrait} for the notation. For the next step let us recall the substitution for $d\psi$ given in~\eqref{Deriva-of-Omega/and/(l+q)},

\begin{equation}\label{Substitution-for-d-psi/Repro}
   d\psi=\frac{l+q}{4b\eta\sin\psi}d(b\ln\omega)\,.
\end{equation}

If $k=2$, then the integration by parts applied to \eqref{Vanishing-RE-Final/Repro} yields
\begin{equation}
    \varphi_\lambda(r(\eta))
    =\frac{\rho^2-\eta^2}{2\eta\rho b} \sin\left(b\ln\frac{\rho+\eta}{\rho-\eta}\right)
    =\frac{\rho^2-\eta^2}{2\eta\rho b} \sin\left(\frac{br}{\rho}\right)\,.
\end{equation}
Therefore,
\begin{equation}
    |b|\leq\frac{1}{|\varphi_\lambda(r(\eta))|}\frac{\rho^2-\eta^2}{2\eta\rho}
    =\frac{1}{\sinh(r/\rho)}\frac{1}{|\varphi_\lambda(r)|}
    =\frac{T_2(r,\kappa)}{|\varphi_\lambda(r)|}\,,
\end{equation}
where $T_2(r,\kappa)=1/\sinh(r/\rho)$. Now, using the expression for $b$, we obtain the following upper bound for $\lambda$.
\begin{equation}
    \lambda\leq-\frac{\kappa k^2}{4}-\kappa \left(\frac{T_2(r,\kappa)}{|\varphi_\lambda(r)|}\right)^2\,.
\end{equation}

If $k\geq 3$, then the same integration by parts applied to the integral in \eqref{Vanishing-RE-Final/Repro}, yields
\begin{equation}
    \varphi_\lambda(r(\eta))
    =\frac{4\rho(\rho^2-\eta^2)^{k/2}\sigma_{k-1}}{|S(\rho)|}\frac{-1}{4b\eta}
    \int\limits_0^{\pi/2}
    \sin(b\ln\omega)d\sin^{k-2}\psi\,,
\end{equation}
and therefore,
\begin{equation}
    |\varphi_\lambda(r)|
    \leq\frac{4\rho(\rho^2-\eta^2)^{k/2}\sigma_{k-1}}{|S(\rho)|}\frac{k-2}{4b\eta}\frac{\pi}{2}
    =\frac{T_3(\kappa,k,r)}{|b|}\,,
\end{equation}
where $T_3=(k-2)\pi\rho(\rho^2-\eta^2)^{k/2}\sigma_{k-1}/(2\eta|S(\rho)|)$. Thus, as above, we conclude
\begin{equation}
    |b|\leq\frac{T_3(\kappa,k,r)}{|\varphi_\lambda(r)|}
    \quad\text{and}\quad
    \lambda\leq -\frac{\kappa  k^2}{4}-\kappa\left(\frac{T_3(\kappa,k,r)}{|\varphi_\lambda(r)|}\right)^2\,.
\end{equation}

The most difficult upper bound happens to be for $k=1$. In this case integration by parts does not help and we need to develop a technique motivated by the proof of Stationary Phase Theorem from~\cite{Olver}, (p.101, Theorem 13.1). In this case formula \eqref{Vanishing-RE-Final/Repro} yields
    \begin{equation}
        \frac{\pi}{4}\frac{\varphi_\lambda(r)}{\sqrt{\rho^2-\eta^2}}
        =\int\limits_0^{\pi/2}
        \frac{\cos(b\ln\omega)}{l+q}d\psi
        =\Re{\int\limits_0^{\pi/2}\frac{e^{ib\ln\omega}}{l+q}d\psi}\,.
    \end{equation}

Therefore, using the substitution for $d\psi$, we can write the following estimate for $\varphi(r)$.
    \begin{equation}\label{Varphi_Estimate_One_Dim}
        \frac{\pi}{4}\frac{|\varphi_\lambda(r)|}{\sqrt{\rho^2-\eta^2}}\leq
        \left|\int\limits_0^{\pi/2}\frac{e^{ib\ln\omega}}{l+q}
        \frac{l+q}{4\eta\sin\psi}d\ln\omega\right|
        =\frac{1}{4\eta}\left|\int\limits_0^{\pi/2}\frac{e^{ib\tau}}{\sin\psi}d\tau\right|
        =\frac{|I|}{4\eta} ,
    \end{equation}
where $\tau=\ln\omega-\ln(\omega(0))$, and $I$ is the last integral in \eqref{Varphi_Estimate_One_Dim}. The direct calculation shows that
    \begin{equation}
        \tau=\frac{\eta}{\rho}\psi^2+C\psi^4+o(\psi^4)\,,
    \end{equation}
where $C$ is some constant. The last formula yields
    \begin{equation}
        \lim\limits_{\psi\rightarrow 0}
        \left\{\frac{\psi}{\sin\psi}\frac{\sqrt{\tau}}{\psi}\right\}
        =
        \lim\limits_{\tau\rightarrow 0}
        \left\{\frac{\psi}{\sin\psi}\frac{\sqrt{\tau}}{\psi}\right\}
        =\lim\limits_{\tau\rightarrow 0}
        \left\{\frac{\sqrt{\tau}}{\sin\psi}\right\}
        =\sqrt{\frac{\eta}{\rho}}.
    \end{equation}

Therefore, the function $\sqrt{\tau}/\sin\psi$ is bounded for $\psi\in[0,\pi/2]$. Let us find the upper bound. First, let us show that it is a monotone function and therefore, it assumes its maximum value at $\psi=\pi/2$. Indeed, its derivative
    \begin{equation}
        \frac{d}{d\psi}\left( \frac{\sqrt{\tau}}{\sin\psi}  \right)
        =\frac{\tau^{-1/2}\tau'_\psi\sin\psi-2\sqrt{\tau}\cos\psi}{2\sin^2\psi}
    \end{equation}
remains positive as long as
    \begin{equation}
        \tau^{-1/2}\tau'_\psi\sin\psi> 2\sqrt{\tau}\cos\psi
        \quad\text{or}\quad
        2\eta\tan\psi\sin\psi>\tau(l+q),
    \end{equation}
since $\tau'_\psi=4\eta\sin\psi/(l+q)$. Note now that left hand side and right hand side expressions in the last inequality are equal to $0$, when $\psi=0$. For $\psi\in(0,\pi/2]$, the derivative of the left hand side expression is greater than the derivative of the right hand side expression, since
    \begin{equation}
        \frac{d}{d\psi}\{2\eta\tan\psi\sin\psi\}=2\eta\sin\psi\left(1+\frac{1}{\cos^2\psi} \right),
    \end{equation}

    \begin{equation}
        \frac{d}{d\psi}\{\tau(l+q)\}=\frac{4\eta\sin\psi}{l+q}(l+q)+\tau(l+q)'_\psi,
    \end{equation}
and $(l+q)'_{\psi}<0$ for $\psi\in(0,\pi/2]$. Therefore,
    \begin{equation}\label{UpperBoundForSqrt(lnOmega)/sinPsi}
        \max\limits_{\psi\in[0,\pi/2]}\left\{ \frac{\sqrt{\tau}}{\sin\psi} \right\}
        =\left.\frac{\sqrt{\tau}}{\sin\psi}\right|_{\psi=\pi/2}
        =\left(\ln\frac{\rho+\eta}{\rho-\eta}\right)^{1/2}
        =\sqrt{\frac{r}{\rho}}
    \end{equation}

%Therefore, there must exist a positive number $T$ such that
%    \begin{equation}\label{Estimate_On_[0_T]}
%         \frac{\psi}{\sin\psi}\frac{\sqrt{\tau}}{\psi}
%        \leq 2\sqrt{\eta/\rho} \quad\text{for every}\quad \tau\in(0,T].
%    \end{equation}
Let us choose $b\geq 2/\pi$, which is equivalent to $1/b\leq \pi/2$. The next step is to split the last integral in \eqref{Varphi_Estimate_One_Dim} into two parts,
    \begin{equation}
        I=\int\limits_{\psi=0}^{\psi=\pi/2}\frac{e^{ib\tau}}{\sin\psi}d\tau
        =\int\limits_{\tau=0}^{\tau=1/b}\frac{e^{ib\tau}}{\sin\psi}d\tau
        + \int\limits_{\tau=1/b}^{\psi=\pi/2}\frac{e^{ib\tau}}{\sin\psi}d\tau
        =I_1+I_2\,,
    \end{equation}
where $I_1$ and $I_2$ are the second and the third integrals respectively in the last formula above. Let us estimate both of that integrals. The estimate in \eqref{UpperBoundForSqrt(lnOmega)/sinPsi} yields
    \begin{equation}
        |I_1|=\left|\int\limits_0^{1/b}e^{ib\tau}
        \frac{\sqrt{\tau}}{\sin\psi}\tau^{-1/2}d\tau\right|
        \leq \sqrt{\frac{r}{\rho}}\int_0^{1/b}\tau^{-1/2}d\tau
        =\sqrt{\frac{r}{\rho}}\frac{1}{\sqrt{b}}.
    \end{equation}
To estimate $|I_2|$, we need to apply integration by parts and then, proceed estimation using \eqref{UpperBoundForSqrt(lnOmega)/sinPsi}.
    \begin{equation}
        I_2=\frac{1}{ib}\int\limits_{\tau=1/b}^{\psi=\pi/2}(e^{ib\tau})'_{\tau}
        \frac{d\tau}{\sin\psi}
        =\frac{1}{ib}
        \left\{\left.\frac{e^{ib\tau}}{\sin\psi}\right|_{\tau=1/b}^{\psi=\pi/2}
        -\int\limits_{\tau=1/b}^{\psi=\pi/2}
        e^{ib\tau}\left(\frac{1}{\sin\psi}\right)^{'}_{\psi}d\psi\right\}.
    \end{equation}
The maximal value computed in \eqref{UpperBoundForSqrt(lnOmega)/sinPsi} yields
    \begin{equation}
        \left|\left.\frac{e^{ib\tau}}{\sin\psi}\right|_{\tau=1/b}^{\psi=\pi/2}\right|
        \leq 1+\frac{\sqrt{\tau}}{\sin\psi}
        \left.\frac{1}{\sqrt{\tau}}\right|_{\tau=1/b}
        \leq 1+ \sqrt{\frac{r}{\rho}}\sqrt{b},
    \end{equation}
and similarly,
    \begin{equation}
        \left|\int\limits_{\tau=1/b}^{\psi=\pi/2}
        e^{ib\tau}\left(\frac{1}{\sin\psi}\right)^{'}_{\psi}d\psi\right|
        \leq \left|\left.\frac{1}{\sin\psi}\right|_{\tau=1/b}^{\psi=\pi/2}\right|
        \leq \sqrt{\frac{r}{\rho}}\sqrt{b}-1.
    \end{equation}
Therefore,
    \begin{equation}
        |I_2|\leq \frac{2}{\sqrt{b}}\sqrt{\frac{r}{\rho}} \quad\text{and}\quad
        |I|\leq |I_1|+|I_2| \leq \frac{3}{\sqrt{b}}\sqrt{\frac{r}{\rho}}\,,
    \end{equation}
which remains true for $b\geq 2/\pi$. This estimate for $I$ combined with \eqref{Varphi_Estimate_One_Dim}, yields
    \begin{equation}
        |\varphi_\lambda(r)|\leq\frac{\sqrt{\rho^2-\eta^2}}{\eta\pi}|I|\leq
        \frac{\sqrt{\rho^2-\eta^2}}{\eta\pi}\sqrt{\frac{r}{\rho}}\frac{3}{\sqrt{b}}
        \quad\text{for}\quad b\geq\frac{2}{\pi}.
    \end{equation}
And now we are ready for the conclusion. For any value $\varphi_{\lambda_0}(r)=\varPhi_0\neq0$, we can choose a large enough $b$ to satisfy the following sequence of inequalities,
\begin{equation}
        |\varphi_\lambda(r)|\leq
        \frac{\sqrt{\rho^2-\eta^2}}{\eta\pi}\sqrt{\frac{r}{\rho}}\frac{3}{\sqrt{b}}
        \leq|\varPhi_0|,
    \end{equation}
where the first inequality will be satisfied for $b\geq2/\pi$. Thus, if we choose
    \begin{equation}
        b>\max\left\{\frac{2}{\pi},
        \frac{\rho^2-\eta^2}{\rho(\eta\pi)^2}\frac{9r}{\varPhi_0^2}\right\},
    \end{equation}
then the absolute value of any eigenfunction $\varphi_\lambda(r)$ will be strictly smaller than the given value $|\varPhi_0|$. Therefore, to satisfy the equation $\varphi_\lambda(r)=\varPhi_0$, we need
    \begin{equation}
        b\leq\max\left\{\frac{2}{\pi},
        \frac{\rho^2-\eta^2}{\rho(\eta\pi)^2}\frac{9r}{\varPhi_0^2}\right\}
        \quad\text{or}\quad
        \lambda\leq -\frac{\kappa k^2}{4}-\kappa \left\{\max\left\{\frac{2}{\pi}
        \frac{\rho^2-\eta^2}{\rho(\eta\pi)^2}\frac{9r}{\varPhi_0^2}\right\}
        \right\}^2.
    \end{equation}
%or, equivalently,
%    \begin{equation}
%        \lambda\leq -\frac{\kappa k^2}{4}-\kappa \left\{\max\left\{\frac{2}{\pi},
%        \frac{\rho^2-\eta^2}{\rho(\eta\pi)^2}\frac{9r}{\varPhi_0^2}\right\}
%        \right\}^2.
%    \end{equation}
This completes the proof of Theorem~\ref{Thm-UpperBoundEgenvalue}.
\end{proof}

\begin{theorem}[One Radius Theorem for a Bounded Eigenvalue]
     Let $\varphi_\lambda(r)$ be a real radial eigenfunction assuming value one at the origin with     $-\kappa k^2/4<\lambda\leq\Lambda=-\kappa(p^2+k^2/4)$, where $p$ is some positive number. Then, such an eigenvalue $\lambda$ can be uniquely restored if we know the value of this eigenfunction at some point $r\leq\pi\rho/p$.
\end{theorem}

\begin{proof}
    Note that to proof the theorem it is enough to show that $\varphi_\lambda(r)$ is a strictly decreasing function with respect to the real eigenvalue $\lambda$ for every fixed $r\leq\pi\rho/p$.
   Recall that for a radial eigenfunction for $\lambda>-\kappa k^2/4$ can be represented by the integral formula in~\eqref{Vanishing-RE-Final}. Differentiation with respect to $b$ under the integral sign yields
    \begin{equation}\label{Deriva-of-VanishingEigenfcnIntegral}
        \frac{d\varphi_\lambda(r)}{db}
        =\frac{4\rho(\rho^2-\eta^2)^{k/2}\sigma_{k-1}}{|S(\rho)|}
        \int\limits_0^{\pi/2}
        \frac{\sin^{k-1}\psi}{l+q}\ln\frac{l}{q}\left(-\sin\left(b\ln\frac{l}{q}\right)\right)d\psi\,,
    \end{equation}
Note now that the lemma's conditions imply $b\leq p$, $r\leq\pi\rho/p$ and then $br/\rho\leq\pi$. Therefore,
    \begin{equation}
        -\pi\leq-\frac{br}{\rho}
        =-b\ln\frac{\rho+\eta}{\rho-\eta}
        =b\ln\frac{\rho-\eta}{\rho+\eta}
        \leq b\ln\frac{l}{q}\leq0.
    \end{equation}
Thus for $\psi\in(0,\pi/2)$, we have
     \begin{equation}
        -\sin b\left( \ln\frac{l}{q} \right) >0,
        \quad\text{while}\quad
        \ln\frac{l}{q}<0.
    \end{equation}
This implies
    \begin{equation}
        \frac{d\varphi_\lambda(r)}{db}<0
        \quad\text{and}\quad
        \frac{d\varphi_\lambda(r)}{d\lambda}
        =\frac{d\varphi_\lambda(r)}{db}\frac{db}{d\lambda}<0,
    \end{equation}
since
    \begin{equation}
        b=\sqrt{\frac{\lambda}{-\kappa}-\frac{k^2}{4}}
        \quad\text{and}\quad
        \frac{db}{d\lambda}
        =-\frac{1}{2\kappa}\left( \frac{\lambda}{-\kappa}-\frac{k^2}{4} \right)^{-1/2}>0.
    \end{equation}
This completes the proof of the theorem.
\end{proof}
%
Recall that by corollary~\ref{Cor-Separator-Lambdas}, the inequality in~\eqref{BoundaryFcn} yields the lower bound for the eigenvalue of a radial eigenfunction assuming the value one at the origin and failing the inequality in~\eqref{BoundaryFcn}. This lower bound is $-\kappa k^2/4$. On the other hand, according to theorem~\ref{Thm-UpperBoundEgenvalue}, the upper bound for such an eigenvalue can be computed using any non zero value of $\varphi_\lambda(r)$ at any point $r>0$. Therefore, in any case, to restore the eigenvalue, we need at most two non-zero values of the given radial eigenfunction.
This observation can be summarized by following theorem.

\begin{theorem}[Two Radii Theorem for the eigenvalue of a vanishing real eigenfunction]\label{TwoRad_SmallValues}
     Let $\varphi_\lambda(r)$ be a real radial eigenfunction assuming value $1$ at the origin, but fails the inequality in formula~\ref{BoundaryFcn} at least for one point $r_1$, or equivalently, this function assumes zero for some finite point. It is clear that the radial eigenfunction is not identically zero and we assume $\varphi_\lambda(r_1)\neq 0$ at some point $r_1$. Then, to restore the corresponding eigenvalue uniquely, it is enough to compute an additional value of $\varphi_\lambda(r_0)$ at a point $r_0\in(0,\pi\rho/p]$, where $p=p(\kappa, k, r_1, |\varphi_\lambda(r_1)|)$ is the upper bound for $b$ computed in theorem~\ref{Thm-UpperBoundEgenvalue}.
\end{theorem}

\bibliographystyle{IEEEtran}

\begin{thebibliography}{99}
\addcontentsline{toc}{chapter}{\numberline{}Bibliography }

\bibitem{Flatto} L.Flatto, \textbf{The Converse of Gauss’s Theorem for Harmonic Functions}, Journal of Differential Equations, vol. 1, issue 4, pp. 483-490, 1965.
\bibitem{Delsarte} J. Delsarte, \textbf{Lectures on Topics in Mean Periodic Functions And The Two-Radius Theorem}, Tata Institute of Fundamental Research, Bombay, 1961.
\bibitem{Koranyi} A. Kor\'anyi, \textbf{On Mean Value Property for Hyperbolic Spaces}, Contemporary Mathematics, vol. 191, pp. 107-116, 1995.
\bibitem{Schwarz} H.A. Schwarz, \textbf{Gesammelte Mathematische Abhandlungen}, Chelsea Publishing Company Bronx, New York, 1972, (pp. 359-361).
\bibitem{Ahlfors} Lars V. Ahlfors, \textbf{Complex Analysis}, McGraw-Hill, Inc., 1966, (pp. 168-169).
\bibitem{Reinhold} Reinhold Remmert, \textbf{Classical Topics in Complex Function Theory}, Springer-Verlag New York, Inc., 1998.
\bibitem{Chavel} Isaac Chavel, \textbf{Eigenvalues in Riemannian Geometry}, Academic Press, 1984.
\bibitem{Olver} W.J.Olver, \textbf{Asymptotics and Special Functions}, Academic Press, Inc., New York and London, 1974.
\bibitem{Hartman} Philip Hartman, \textbf{Ordinary Differential Equations}, John Wiley and Sons, Inc., 1964.
\bibitem{Simmons} George F. Simmons, \textbf{Differential Equations with Applications and Historical Notes}, McGraw-Hill, Inc., 1972.
\bibitem{Artamoshin} S. Artamoshin, \emph{Lower bounds for the first Dirichlet eigenvalue of the Laplacian for domains in hyperbolic space}, \textbf{Mathematical Proceedings of the Cambridge Philosophical Society}, vol. 160, issue 02, pp. 191-208, 2016.
%\bibitem{Artamoshin_Diss} S.Artamoshin, \textbf{Geometric Interpretation of the 2-D Poisson Kernel and Its Applications}, http://arxiv.org/pdf/0912.0223v1.pdf, 2009.
%\bibitem{Artamoshin_Sphe_Ratio} S.Artamoshin, \textbf{The Spherical Ratio of Two Points
%   and Its Integral Properties}, http://arxiv.org/pdf/1411.1915.pdf, 2014.

\end{thebibliography}

}
\end{spacing}

\end{document}